\documentclass[11pt,letterpaper]{amsart}
\usepackage{amsmath,amsthm}
\usepackage{amstext,amsfonts,bm,amssymb}
\usepackage{graphics,graphicx} 
\usepackage{a4wide}
\usepackage{float}
\usepackage{color}
\usepackage{verbatim}

\usepackage{gastex}

\theoremstyle{plain}
\newtheorem{theorem}{Theorem}
\newtheorem{Prop}{Proposition}
\newtheorem{Cor}{Corollary}

\def\ang#1{\left\langle#1\right\rangle}

\def\sttf2#1#2{\left[\!\!\left[#1\atop#2\right]\!\!\right]}  
\def\stf3f#1#2{\left[\!\!\left[\!\!\left[#1\atop#2\right]\!\!\right]\!\!\right]} 
\def\stff4#1#2{\left[\!\!\left[\!\!\left[\!\!\left[#1\atop#2\right]\!\!\right]\!\!\right]\!\!\right]}
\def\stss2#1#2{\left\{\!\!\left\{#1\atop#2\right\}\!\!\right\}}

\allowdisplaybreaks 

\begin{document}

\title[Finite $q$-multiple harmonic sums at roots of unity]{Finite $q$-multiple harmonic sums at roots of unity: Symmetrized identities for $1$-$2$ index multisets}  

\author{Zikang Dong}
\address{School of Mathematical Sciences\\ Soochow University\\ Suzhou 215006\\ China} 
\email{zikangdong@gmail.com} 

\author{Takao Komatsu}
\address{
Department of Mathematics, Institute of Science Tokyo, 2-12-1 Ookayama, Meguro-ku, Tokyo 152-8551, Japan}
\email{komatsu.t.al@m.titech.ac.jp;\,komatsu@zstu.edu.cn}
\thanks{T.K. is the corresponding author.}

\date{
}

\begin{abstract}
Closed-form results for finite $q$-multiple harmonic sums are abundant for uniformly repeated indices, whereas mixed-weight tuples pose substantial combinatorial challenges. Building on recent progress for patterns containing a single weight-$2$ entry embedded among weight-$1$ indices, we treat multisets composed of arbitrary numbers of $1$s and $2$s at roots of unity. Since individual ordered sums resist simple evaluation, we analyze their symmetrized sum over all permutations. Using newly derived evaluations for $q$-multiple sums with negative powers together with multiplicative identities, we obtain compact binomial-based closed forms. These enlarge the repertoire of explicitly computable finite $q$-multiple zeta-type objects and lay groundwork for further mixed-index investigations.
\medskip

\end{abstract}


\maketitle

\section{Introduction}\label{sec:1}

Multiple zeta values occupy a central position in the study of nested sums and have been connected with a remarkably broad range of topics in number theory, combinatorics, and special functions. Among their numerous $q$-analogues, finite $q$-multiple harmonic sums provide a particularly fruitful setting in which algebraic and combinatorial structures can be investigated simultaneously. In this paper, we focus on finite $q$-multiple harmonic sums at roots of unity and, more specifically, on the extent to which their values admit explicit formulas when the indices are not uniform.

For positive integers $s_1,s_2,\ldots,s_m$, consider
\begin{equation}
\mathfrak Z_n(q;;s_1,s_2,\ldots,s_m)
:=\sum_{1\le i_1<i_2<\cdots<i_m\le n-1}
\frac{1}{(1-q^{i_1})^{s_1}(1-q^{i_2})^{s_2}\cdots
(1-q^{i_m})^{s_m}}.
\label{def:qssmzv}
\end{equation}
When all the indices coincide, say $s_1=\cdots=s_m=s$, previous papers
(\cite{CK,Ko25b,Ko25a,KL,KP,KW}) used the simpler notation
$$ 
\mathfrak Z_n(q;m,s)
:=\mathfrak Z_n(q;;\underbrace{s,\ldots,s}_m).
$$ 
The present paper is concerned with the substantially richer situation in which the indices are allowed to vary. Accordingly, the number of summation indices is left implicit in the notation in \eqref{def:qssmzv}.

The infinite counterpart of \eqref{def:qssmzv},
\begin{equation}
\mathfrak Z(q;;s_1,s_2,\ldots,s_m)
:=\sum_{1\le i_1<i_2<\cdots<i_m}
\frac{1}{(1-q^{i_1})^{s_1}(1-q^{i_2})^{s_2}\cdots
(1-q^{i_m})^{s_m}},
\label{def:qsmzv}
\end{equation}
was studied by Schlesinger \cite{Schlesinger}, and several related forms and
generalizations have subsequently been investigated; see, for example,
\cite{Bradley,OOZ,Zhao,Zudilin}. Finite versions have also attracted considerable attention; see \cite{BTT18,BTT20,Takeyama09,Tasaka21}. Moreover, after
multiplication by $(1-q)^m$ and passage to the limit $q\to1$, \eqref{def:qsmzv}
leads naturally to the multiple zeta value
\begin{equation}
\zeta(s_1,s_2,\ldots,s_m)
:=\sum_{1\le i_1<i_2<\cdots<i_m}
\frac{1}{i_1^{s_1}i_2^{s_2}\cdots i_m^{s_m}},
\label{def:mzv}
\end{equation}
which is one of the fundamental objects in the theory of multiple zeta values
and their generalizations; see, for instance, \cite{Zhao_book}.

A particularly tractable and important specialization is obtained by taking
$q=\zeta_n$, where
$
\zeta_n=e^{2\pi\sqrt{-1}/n}
$
is a primitive $n$th root of unity. In this setting, the remarkable algebraic
relations among the quantities $1-\zeta_n^j$ often lead to unexpectedly simple
closed forms. When the indices are uniform, a substantial collection of such
formulas is now available. For example, one of the simplest identities is
\begin{equation}
\mathfrak Z_n(\zeta_n;;\underbrace{1,\ldots,1}_m)
=\frac{1}{m+1}\binom{n-1}{m},
\label{eq}
\end{equation}
while, for a fixed positive integer $s$, the value
$\mathfrak Z_n(\zeta_n;;\underbrace{s,\ldots,s}_m)$ can be expressed in terms of
determinants; in particular, for $m=s$ one obtains
\begin{equation}
\mathfrak Z_n(\zeta_n;;s)=
\left|\begin{array}{ccccc}
\frac{n-1}{2}&1&0&\cdots&\\
\frac{2}{3}\binom{n-1}{2}&\frac{n-1}{2}&1&&\vdots\\
\vdots&&\ddots&&0\\
\frac{s-1}{s}\binom{n-1}{s-1}&\frac{1}{s-1}\binom{n-1}{s-2}&\cdots&\frac{n-1}{2}&1\\
\frac{s}{s+1}\binom{n-1}{s}&\frac{1}{s}\binom{n-1}{s-1}&\cdots&\frac{1}{3}\binom{n-1}{2}&\frac{n-1}{2}
\end{array}\right|\,.
\label{eq}
\end{equation}
These formulas, together with related results in
\cite{Ko25b,Ko25a,KL,KP,KW}, illustrate the remarkable degree of explicitness
that can be achieved in the uniform-index case.

The situation changes considerably once the indices are allowed to vary.
Although the individual $s_j$ in \eqref{def:qssmzv} may be arbitrary positive
integers, relatively few explicit formulas are known for genuinely
non-uniform index strings. The principal obstacle is structural: the methods
that are particularly effective for repeated indices, including the
combinatorial properties of Stirling numbers and generating-function
expansions associated with symmetric functions, no longer apply in a direct
fashion. Thus, even the apparently modest first departure from the uniform
case,
$
1-\cdots-1,2,1-\cdots-1,
$
already presents a nontrivial problem. Indeed, although an elegant conjectural
formula was proposed in \cite{BTT20}, its proof required substantial additional
work.  
Recently, in \cite{BIK}, the proof for the case of $1-\cdots-1,2,1-\cdots-1$ was finally completed. That is, for non-negative integers $a$ and $b$, we proved the following.  
For integers $n$ and $m$ with $n,m\ge 2$,  
\begin{equation}
\mathfrak Z_n(\zeta_n;;\underbrace{\underbrace{1,\dots,1}_a,2,\underbrace{1,\dots,1}_b}_m)+\mathfrak Z_n(\zeta_n;;\underbrace{\underbrace{1,\dots,1}_b,2,\underbrace{1,\dots,1}_a}_m)=-\frac{m!(n-2 m-3)}{(m+2)!}\binom{n-1}{m}\,. 
\label{eq:11112}
\end{equation} 
In other words, by taking the real part,  
\begin{equation}
\mathfrak{Re}\left(\mathfrak Z_n(\zeta_n;;\underbrace{\underbrace{1,\dots,1}_a,2,\underbrace{1,\dots,1}_b}_m)\right)=-\frac{m!(n-2 m-3)}{2(m+2)!}\binom{n-1}{m}\,.
\label{eq:11112r}
\end{equation}  
This result suggests a natural hierarchy of problems: what happens when the
relative frequency of the index $2$ is increased? More specifically, can one
still obtain explicit information when several occurrences of $2$ are
interspersed with occurrences of $1$?

The purpose of the present paper is to pursue this question for the family
$$
\underbrace{2,\ldots,2}_{r},
\underbrace{1,\ldots,1}_{m-r},
\qquad 0\le r\le m.
$$
At present, it appears difficult to obtain a comparably simple formula for
each individual
$
\mathfrak Z_n(\zeta_n;;s_1,s_2,\ldots,s_m)
$
with a prescribed ordering of the indices. A more effective approach is
therefore to exploit the symmetry hidden in the collection of all distinct
permutations of the index string. This leads to explicit formulas for the
aggregate, and consequently for the average, of the corresponding
$q$-multiple zeta values.
$$
\{s_1,s_2,\dots,s_m\}=
\{\underbrace{2,\ldots,2}_{r},
\underbrace{1,\ldots,1}_{m-r}\},
$$
and the sum extends over all distinct permutations of this multiset. Thus,
although the ordering of the indices remains essential at the level of an
individual multiple sum, a strikingly simple closed form emerges after
symmetrization. This phenomenon provides a natural extension of the
$1-\cdots-1,2,1-\cdots-1$ case and reveals a broader combinatorial structure
behind these root-of-unity evaluations.

A second ingredient of our approach is to enlarge the class of auxiliary
$q$-multiple zeta values by allowing non-positive powers in
\eqref{def:qssmzv}. More precisely, we consider products of the form
$$
\mathfrak Z_n(\zeta_n;;\underbrace{A,\ldots,A}_m)
\mathfrak Z_n(\zeta_n;;\underbrace{B,\ldots,B}_{\ell}),
$$
where $A$ and $B$ may be positive or negative integers. This extension is not
merely formal: the resulting identities provide an efficient mechanism for
encoding the mixed-index sums considered in this paper. In particular, the
negative-index cases $A=-1$ and $A=-2$ admit especially simple explicit
expressions, which will be established in the next section and subsequently
used in the derivation of our main formulas.

The results below therefore combine three complementary features of the
subject: root-of-unity evaluations, the combinatorics of non-uniform index
strings, and auxiliary identities involving negative powers. Together they
extend the range of finite $q$-multiple zeta values for which explicit
formulas are available and, more importantly, indicate a systematic way of
approaching mixed-index configurations beyond the first non-uniform case.

\section{$q$-multiple zeta values with negative powers}   

 


We have already known the explicit expressions of values of $\mathfrak Z_n(\zeta_n;;\underbrace{A,\dots,A}_m)$ for $A=1,2,3$ through the way of calculations in \cite{Ko25a}.  

Now, we shall see how to calculate $\mathfrak Z_n(\zeta_n;;\underbrace{-B,\dots,-B}_\ell)$, then get the explicit expressions of their values for $B=1,2$.

\begin{theorem}  
For $n\ge \ell$, we have 
\begin{equation}
\mathfrak Z_n(\zeta_n;;\underbrace{-1,\dots,-1}_\ell)=\binom{n}{\ell} 
\label{eq:-1}
\end{equation}
and 
\begin{equation}
\mathfrak Z_n(\zeta_n;;\underbrace{-2,\dots,-2}_\ell)=\binom{n}{\ell}+2(-1)^{n-\ell-1}\binom{n}{2 n-2 \ell}\,.
\label{eq:-2}
\end{equation}
\label{th:-12}
\end{theorem}  
\begin{proof}  
For convenience, put 
\begin{equation}
u_r:=\frac{1}{1-\zeta_n^r}=\frac{1}{2}+\frac{\sqrt{-1}}{2}\cot\frac{r\pi}{n}\quad(1\le r\le n-1)\,.
\label{eq:ur}
\end{equation} 
Consider the generating function 
\begin{align*} 
\sum_{\ell=0}^{n-1}\sum_{1\le i_1<\dots<i_\ell\le n-1}\frac{X^\ell}{u_{i_1}^s\cdots u_{i_\ell}^s}
&=\prod_{j=1}^{n-1}\bigl(1+(1-\zeta_n^j)^s X\bigr)\\
&=X^{n-1}\prod_{j=1}^{n-1}\bigl(X^{-1}+(1-\zeta_n^j)^s\bigr)\,.  
\end{align*}
Assume $\alpha_i$ ($i=1,2,\dots,s$) are the roots such that 
\begin{equation}
\prod_{i=1}^s(\alpha_i-Y)=X^{-1}+(1-Y)^s\,. 
\label{eq:alpha}
\end{equation} 
Since 
$$
\prod_{i=1}^s(\alpha_i-1)=X^{-1}\,,
$$ 
we have  
\begin{align*}
\prod_{j=1}^{n-1}\prod_{i=1}^s(\alpha_i-\zeta_n^j)=\prod_{i=1}^s\frac{\alpha_i^n-1}{\alpha^i-1}=X\prod_{i=1}^s(\alpha_i^n-1)\,. 
\end{align*} 
Thus, 
\begin{align}
&\sum_{n=1}^\infty\frac{Y^n}{n}\sum_{\ell=0}^{n-1}\sum_{1\le i_1<\dots<i_\ell\le n-1}\frac{X^\ell}{u_{i_1}^s\cdots u_{i_\ell}^s}\notag\\
&=\sum_{n=1}^\infty\frac{Y^n}{n}X^{n-1}X\prod_{i=1}^s(\alpha_i^n-1)\notag\\
&=\sum_{n=1}^\infty\frac{(X Y)^n}{n}\left((-1)^s+\sum_{l=1}^s(-1)^{s-l}\sum_{1\le i_1<\dots<i_l\le s}(\alpha_{i_1}\cdots\alpha_{i_l})^n\right)\notag\\
&=(-1)^{s-1}\left(-\sum_{n=1}^\infty\frac{(X Y)^n}{n}+\sum_{l=1}^s(-1)^{l-1}\frac{\sum_{1\le i_1<\dots<i_l\le s}(\alpha_{i_1}\cdots\alpha_{i_l}X Y)^n}{n}\right)\notag\\
&=(-1)^{s-1}\left(\log(1-X Y)+\sum_{l=1}^s(-1)^l\sum_{1\le i_1<\dots<i_l\le s}\log(1-\alpha_{i_1}\cdots\alpha_{i_l}X Y)\right)\notag\\
&=(-1)^{s-1}\log\left(\prod_{l=0}^s F_{s,l}(X,Y)^{(-1)^l}\right)\,,
\label{eq:s123} 
\end{align} 
where $F_{s,l}(X,Y)=\prod_{1\le i_1<\dots<i_l\le s}(1-\alpha_{i_1}\cdots\alpha_{i_l}X Y)$ ($1\le l\le s$) with $F_{0,l}(X,Y)=1-X Y$. 

Now, set $s=1$. By $\alpha_1=1+X^{-1}$ in (\ref{eq:alpha}), we get $F_{1,0}(X,Y)=1-X Y$ and $F_{1,1}(X,Y)=1-(X+1)Y$. Then the right-hand side of (\ref{eq:s123}) is given as 
\begin{align*}
&\log\frac{F_{1,0}(X,Y)}{F_{1,1}(X,Y)}=\log\frac{1-X Y}{1-(X+1)Y}\\
&=-\log\left(1-\frac{Y}{1-X Y}\right)\\
&=\sum_{m=1}^\infty\frac{1}{m}\left(\frac{Y}{1-X Y}\right)^m\\
&=\sum_{m=1}^\infty\frac{Y^m}{m}\sum_{j=0}^\infty\binom{m+j-1}{m-1}X^j Y^j\\
&=\sum_{n=1}^\infty\sum_{\ell=0}^{n-1}\frac{1}{n-\ell}\binom{n-1}{n-\ell-1}X^\ell Y^n\,. 
\end{align*}
Here, we put $\ell=j$ and $n=m+j$. 
Comparing the coefficients on both sides, we can get 
\begin{align*}
\sum_{1\le i_1<\dots<i_\ell\le n-1}\frac{1}{u_{i_1}\cdots u_{i_\ell}}
&=\frac{n}{n-\ell}\binom{n-1}{n-\ell-1}\\
&=\binom{n}{n-\ell}=\binom{n}{\ell}\,. 
\end{align*} 

Next, set $s=2$. By $\alpha_1\alpha_2=1+X^{-1}$ and $\alpha_1\alpha_2=2$ in (\ref{eq:alpha}), we get $F_{2,0}(X,Y)=1-X Y$, $F_{2,1}(X,Y)=(1-X Y)^2+x Y^2$ and $F_{2,2}(X,Y)=1-(X+1)Y$. Then the right-hand side of (\ref{eq:s123}) is given as 
\begin{align*}
&-\log\frac{F_{2,0}(X,Y)F_{2,2}(X,Y)}{F_{2,1}(X,Y)}\\
&=-\log\frac{(1-X Y)\bigl(1-(X+1)Y\bigr)}{(1-X Y)^2+X Y^2}\\
&=-\log\left(1-\frac{Y}{1-X Y}\right)+\log\left(1+\frac{X Y^2}{(1-X Y)^2}\right)\\
&=\sum_{m=1}^\infty\frac{1}{m}\left(\frac{Y}{1-X Y}\right)^m+\sum_{m=1}^\infty\frac{(-1)^{m-1}}{m}\left(\frac{X Y^2}{(1-X Y)^2}\right)^m\\
&=\sum_{m=1}^\infty\frac{Y^m}{m}\sum_{j=0}^\infty\binom{m+j-1}{m-1}X^j Y^j
+\sum_{m=1}^\infty\frac{(-1)^{m-1}X^m Y^{2 m}}{m}\sum_{j=0}^\infty\binom{2 m+j-1}{2 m-1}X^j Y^j\\
&=\sum_{n=1}^\infty\sum_{\ell=0}^{n-1}\frac{1}{n-\ell}\binom{n-1}{n-\ell-1}X^\ell Y^n 
+\sum_{n=2}^\infty\sum_{\ell=0}^{n-1}\frac{(-1)^{n-\ell-1}}{n-\ell}\binom{n-1}{2 n-2\ell-1}X^\ell Y^n\,. 
\end{align*} 
Here, we put $\ell=j$ and $n=m+j$ in the first term and $\ell=m+j$ and $n=2 m+j$ in the second term. 
Comparing the coefficients on both sides, we can get 
\begin{align*}
\sum_{1\le i_1<\dots<i_\ell\le n-1}\frac{1}{u_{i_1}^2\cdots u_{i_\ell}^2}
&=\frac{n}{n-\ell}\binom{n-1}{n-\ell-1}+\frac{(-1)^{n-\ell}n}{n-\ell}\binom{n-1}{2 n-2\ell-1}\\
&=\binom{n}{n-\ell}+2(-1)^{n-\ell-1}\binom{n}{2 n-2\ell}\,. 
\end{align*} 
\end{proof}

\section{Identity between several $q$-multiple zeta values with different indices}  

For convenience, put  
$$
\mathcal Y_n(s_1,\dots,s_m):=\sum_{\sigma\in\ang{s_1,s_2,\dots,s_m}}\mathfrak Z_n(\zeta_n;;\sigma)\,,
$$ 
where the sum is taken over all distinct permutations of $\{s_1,s_2,\dots,s_m\}$. Asssume that $\mathcal Y_n(s_1,\dots,s_m)=0$ if $m=0$. 
 
For example, 
\begin{align*}
\mathcal Y_n(2,2,2,1,1)&=\mathfrak Z_n(\zeta_n;;2,2,2,1,1)+\mathfrak Z_n(\zeta_n;;2,2,1,2,1)+\mathfrak Z_n(\zeta_n;;2,1,2,2,1)\\
&\quad +\mathfrak Z_n(\zeta_n;;1,2,2,2,1)+\mathfrak Z_n(\zeta_n;;2,2,1,1,2)+\mathfrak Z_n(\zeta_n;;2,1,2,1,2)\\
&\quad +\mathfrak Z_n(\zeta_n;;1,2,2,1,2)+\mathfrak Z_n(\zeta_n;;2,1,1,2,2)+\mathfrak Z_n(\zeta_n;;1,2,1,2,2)\\
&\quad +\mathfrak Z_n(\zeta_n;;1,1,2,2,2)\,.
\end{align*}  

The main result of this section is the following.  Here, when the length of the string block is non-positive, such a string does not appear.   

\begin{theorem}  
For integers $n$, $m$ and $\ell$ with $n-1\ge m\ge\ell\ge 0$, we have  
\begin{align*}
&\sum_{j=0}^{n-m-2}\binom{n-m+\ell-3-2 j}{\ell-1-j}\mathcal Y_n(\underbrace{2,\dots,2}_{m-\ell+1+j},\underbrace{1,\dots,1}_{n-m+\ell-3-2 j})\\ 
&=\frac{1}{n(m+1)}\binom{n-1}{m}\left(\binom{n}{\ell}-\binom{m+1}{\ell}\right)\,.
\end{align*}
\label{th:5}
\end{theorem}

\noindent 
{\it Remark.}  
When $\ell=1$, 
\begin{equation}
\mathcal Y_n(\underbrace{2,\dots,2}_{m},\underbrace{1,\dots,1}_{n-m-2})=\frac{n-m-1}{n(m+1)}\binom{n-1}{m}\,.
\label{eq:22m11nm2}
\end{equation}  
In particular, when $m=n-3$, we have  
$$
\sum_{k=1}^{n-3}\mathfrak Z_n(\zeta_n;;\underbrace{2,\dots,2}_{k-1},1,\underbrace{2,\dots,2}_{n-3-k})=\frac{n-1}{n}\,. 
$$ 

When $\ell=2$, 
\begin{align*}
&(n-m-1)\mathcal Y_n(\underbrace{2,\dots,2}_{m-1},\underbrace{1,\dots,1}_{n-m-1})+\mathcal Y_n(\underbrace{2,\dots,2}_{m},\underbrace{1,\dots,1}_{n-m-3})\\
&=\frac{(n+m)(n-m-1)}{2 n(m+1)}\binom{n-1}{m}\,.
\end{align*}
\bigskip

In order to prove Theorem \ref{th:5}, we need the following.  

\begin{Prop}   
For integers $n$, $m$ and $\ell$ with $n-1\ge m\ge\ell\ge 0$, we have  
\begin{align*}
&\mathfrak Z_n(\zeta_n;;\underbrace{1,\dots,1}_m)\mathfrak Z_n(\zeta_n;;\underbrace{-1,\dots,-1}_\ell)\\
&=\binom{n-m+\ell-1}{n-m-1}\mathfrak Z_n(\zeta_n;;\underbrace{1,\dots,1}_{m-\ell})\\
&\quad +\frac{1}{u_1\cdots u_{n-1}}\sum_{j=0}^{n-m-2}\binom{n-m+\ell-3-2 j}{\ell-1-j}\mathcal Y_n(\underbrace{2,\dots,2}_{m-\ell+1+j},\underbrace{1,\dots,1}_{n-m+\ell-3-2 j})\,. 
\end{align*}
\label{prp:5}
\end{Prop}
\begin{proof}
The first term on the right-hand side is equal to the number of cases where $m$ elements in $\mathfrak Z_n(\zeta_n;;\underbrace{1,\dots,1}_m)$ are chosen from $n-1$ elements and $\ell$ elements in $\mathfrak Z_n(\zeta_n;;\underbrace{-1,\dots,-1}_\ell)$ are removed from them. That is, 
$$
\{u_{j_1},\dots,u_{j_\ell}\}\subseteq\{u_{k_1},\dots,u_{k_m}\}\subseteq\{u_1,\dots,u_{n-1}\}\,.
$$ 
There are 
$$
\binom{n-1}{m}\binom{m}{\ell}
$$ 
ways to choose $m$ items from $n-1$ and then choose $\ell$ items from those $m$ items. 
However, the particular ($m-l$)-tuple $U:=\{u_{i_1},\dots,u_{i_{m-\ell}}\}$ can be chosen in 
$$
\binom{n-1}{m-\ell}
$$ 
different ways, so there are a total of 
$$ 
\left.\binom{n-1}{m}\binom{m}{\ell}\middle/\binom{n-1}{m-\ell}=\binom{n-m+\ell-1}{n-m-1}\right.
$$
such $U$'s. 

The second term on the right-hand side yields when $\{u_{j_1},\dots,u_{j_\ell}\}\not\subset\{u_{k_1},\dots,u_{k_m}\}$. 
Assume that out of the $m$ elements in $\mathfrak Z_n(\zeta_n;;\underbrace{1,\dots,1}_m)$, $r$ elements are not shared with the elements of $\mathfrak Z_n(\zeta_n;;\underbrace{-1,\dots,-1}_\ell)$. Then, $m-r$ elements are shared with the elements of $\mathfrak Z_n(\zeta_n;;\underbrace{-1,\dots,-1}_\ell)$. Furthermore, the remaining $\ell-m+r(>0)$ elements of $\mathfrak Z_n(\zeta_n;;\underbrace{-1,\dots,-1}_\ell)$ are not included in the elements of $\mathfrak Z_n(\zeta_n;;\underbrace{1,\dots,1}_m)$. 
Since there are $n-\ell-r-1$ elements among the $n-1$ elements that do not appear in the $m$ elements and the $\ell$ elements, the indices (powers) when divided by $u_1\dots u_{n-1}$, regardless of order, are $\underbrace{2,\dots,2}_{r}\underbrace{1,\dots,1}_{n+m-\ell-2 r-1}$.  
In other words, $r$ elements appear twice, while $m-r$ elements that are duplicated in both multiple sums and $n-\ell-r-1$ elements that do not appear in either multiple sum appear only once. The $\ell-m+r(>0)$ elements that appear only in the second multiple sum do disappear.  

For example, if $\{u_{k_1},\dots,u_{k_m}\}=\{u_1,u_2,u_3,u_4,u_5,u_6\}$ and $\{u_{j_1},\dots,u_{j_\ell}\}=\{u_4,u_5,u_6,u_7\}$ with $n=10$, then 
$$
\frac{u_1 u_2 u_3 u_4 u_5 u_6}{u_4 u_5 u_6 u_7}=\frac{u_1 u_2 u_3}{u_7}=\frac{u_1^2 u_2^2 u_3^2 u_4 u_5 u_6 u_8 u_9}{u_1\dots u_9}\,,
$$
yielding indices are $22211111$. 

Considering the condition, the range of $r$ is given by 
$$
m-\ell+1\le r\le n-\ell-1\,. 
$$ 
Hence, we can set $r=m-\ell+1+j$ ($0\le j\le n-m-2$).  Hence, the indices (powers) when divided by $u_1\dots u_{n-1}$, regardless of order, are given by 
$$
\underbrace{2,\dots,2}_{m-\ell+1+j},\underbrace{1,\dots,1}_{n-m+\ell-3-2 j}\quad(0\le j\le n-m-2)\,.  
$$ 
For a fixed $j$, the number of cases with such a sequence of indices of $2$s and $1$s depends on the way of choosing $m-r=\ell-1-j$ ones that are common to both sums from the total number of ones $n+m-\ell-2 r-1=n-m+\ell-3-2 j$. In other words, there are 
$$
\binom{n-m+\ell-3-2 j}{\ell-1-j}
$$ 
ways. From the above, the second term on the right-hand side is obtained. 
\end{proof}

\begin{proof}[Proof of Theorem \ref{th:5}.]
Using the identity (\ref{eq:zz-1m}), the identity in Theorem \ref{th:-12} (\ref{eq:-1}) and $u_1\cdots u_{n-1}=1/n$, 
from Proposition \ref{prp:5}, we have  
\begin{align*}
&\sum_{j=0}^{n-m-2}\binom{n-m+\ell-3-2 j}{\ell-1-j}\mathcal Y_n(\underbrace{2,\dots,2}_{m-\ell+1+j},\underbrace{1,\dots,1}_{n-m+\ell-3-2 j})\\
&=\frac{1}{n}\left(\frac{1}{m+1}\binom{n-1}{m}\binom{n}{\ell}-\binom{n-m+\ell-1}{\ell}\frac{1}{m-\ell+1}\binom{n-1}{m-\ell}\right)\\
&=\frac{1}{n(m+1)}\binom{n-1}{m}\left(\binom{n}{\ell}-\binom{m+1}{\ell}\right)\,.
\end{align*}
\end{proof}

By the same method of Theorem \ref{th:5}, we can also show the following.  Here, $A$ is a positive integer and $A'=2 A$.  

\begin{Cor}  
\begin{align*}
&\sum_{j=0}^{n-m-2}\binom{n-m+\ell-3-2 j}{\ell-1-j}\mathcal Y_n(\underbrace{A',\dots,A'}_{m-\ell+1+j},\underbrace{A,\dots,A}_{n-m+\ell-3-2 j})\\ 
&=\frac{1}{n(m+1)}\binom{n-1}{m}\left(\binom{n}{\ell}-\binom{m+1}{\ell}\right)\,.
\end{align*}
\label{cor:5}
\end{Cor}

\section{More variations}

Consider the product of two multiple harmonic sums: 
$$
\mathfrak Z_n(\zeta_n;;\underbrace{2,\dots,2}_m)\mathfrak Z_n(\zeta_n;;\underbrace{-1,\dots,-1}_\ell)\,.
$$  
Then in a similar way to the proof of Proposition \ref{prp:5}, we can get the results in multiple zeta values that also contain powers (indexes) of $3$.

\begin{theorem}  
For integers $n$, $m$ and $\ell$ with $n-1\ge m\ge\ell\ge 0$, we have  
\begin{align*}
&\mathcal Y_n(\underbrace{2,\dots,2}_{m-\ell},\underbrace{1,\dots,1}_{\ell})
+n\sum_{j=0}^{n-m-2}\mathcal Y_n(\underbrace{3,\dots,3}_{m-\ell+1+j},\underbrace{2,\dots,2}_{\ell-1-j},\underbrace{1,\dots,1}_{n-m-2-j})\\
&=\frac{1}{n(m+1)}\left(\binom{n-1}{m}+(-1)^m\binom{n-1}{2 m+1}\right)\binom{n}{\ell}\,. 
\end{align*} 
\label{th:7}
\end{theorem}

\noindent 
{\it Remark.}  
When $\ell=1$ in Theorem \ref{th:7}, we have 
$$
\mathcal Y_n(\underbrace{2,\dots,2}_{m-1},1)
+n\mathcal Y_n(\underbrace{3,\dots,3}_{m},\underbrace{1,\dots,1}_{n-m-2})\\
=\frac{1}{m+1}\left(\binom{n-1}{m}+(-1)^m\binom{n-1}{2 m+1}\right)\,. 
$$

Theorem \ref{th:7} is the immediate consequence of the following.  

\begin{Prop}   
For integers $n$, $m$ and $\ell$ with $n-1\ge m\ge\ell\ge 0$, we have   
\begin{align*}
&\mathfrak Z_n(\zeta_n;;\underbrace{2,\dots,2}_m)\mathfrak Z_n(\zeta_n;;\underbrace{-1,\dots,-1}_\ell)\\
&=\mathcal Y_n(\underbrace{2,\dots,2}_{m-\ell},\underbrace{1,\dots,1}_{\ell})\\
&\quad +\frac{1}{u_1\cdots u_{n-1}}\sum_{j=0}^{n-m-2}\mathcal Y_n(\underbrace{3,\dots,3}_{m-\ell+1+j},\underbrace{2,\dots,2}_{\ell-1-j},\underbrace{1,\dots,1}_{n-m-2-j})\,. 
\end{align*}
\label{prp:7}
\end{Prop}
\begin{proof}
Similarly to the proof of Proposition \ref{prp:5}, there are two situations: $\{u_{j_1},\dots,u_{j_\ell}\}\subset\{u_{k_1},\dots,u_{k_m}\}$ and $\{u_{j_1},\dots,u_{j_\ell}\}\not\subset\{u_{k_1},\dots,u_{k_m}\}$. 
The first case is to choose $m$ items from $n-1$ and then choose $\ell$ items from those $m$ items. This yields the term $\mathcal Y_n(\underbrace{2,\dots,2}_{m-\ell},\underbrace{1,\dots,1}_{\ell})$
For example, if $\{u_{k_1},\dots,u_{k_m}\}=\{u_1,u_2,u_3,u_4,u_5,u_6\}$ and $\{u_{j_1},\dots,u_{j_\ell}\}=\{u_4,u_5\}$ with $n=10$, then 
$$
\frac{u_1^2 u_2^2 u_3^2 u_4^2 u_5^2 u_6^2}{u_4 u_5}=u_1^2 u_2^2 u_3^2 u_4 u_5 u_6^2\,,
$$
yielding indices are $222112$. 

In the second case, for example, if $U_1:=\{u_{k_1},\dots,u_{k_m}\}=\{u_1,u_2,u_3,u_4,u_5,u_6\}$ and $U_2:=\{u_{j_1},\dots,u_{j_\ell}\}=\{u_4,u_5,u_6,u_7\}$ with $n=10$, then 
$$
\frac{u_1^2 u_2^2 u_3^2 u_4^2 u_5^2 u_6^2}{u_4 u_5 u_6 u_7}=\frac{u_1^2 u_2^2 u_3^2 u_4 u_5 u_6}{u_7}=\frac{u_1^3 u_2^3 u_3^3 u_4^2 u_5^2 u_6^2 u_8 u_9}{u_1\dots u_9}\,,
$$
yielding indices are $33322211$. 
If $r=\#(U_1\backslash U_2)$, then $\#(U_1\cap U_2)=m-r$ with $\ell-m+r\ge 1$. In addition, $n-\ell-r-1$ elements are neither in $U_1$ nor in $U_2$. Hence, this gives  
$$
\frac{1}{u_1\cdots u_{n-1}}\mathcal Y_n(\underbrace{3,\dots,3}_{r},\underbrace{2,\dots,2}_{m-r},\underbrace{1,\dots,1}_{n-\ell-r-1})\,. 
$$ 
Since $m-\ell+1\le r\le n-\ell-1$, we can write $r=m-\ell+1+j$ ($0\le j\le n-m-2$). Hence, for each $j$, this term can be written as  
$$
\mathcal Y_n(\underbrace{3,\dots,3}_{m-\ell+1+j},\underbrace{2,\dots,2}_{\ell-1-j},\underbrace{1,\dots,1}_{n-m-2-j})\,. 
$$ 
Therefore, we get the desired result.  
\end{proof}

\begin{proof}[Proof of Theorem \ref{th:7}.]
When $u_r$ is given as in (\ref{eq:ur}), by 
$$
\mathfrak Z_n(\zeta_n;;\underbrace{2,\dots,2}_m)=\frac{1}{n(m+1)}\left(\binom{n-1}{m}+(-1)^m\binom{n-1}{2 m+1}\right)
$$ 
(\cite[Theorem 5]{Ko25a}), from Proposition \ref{prp:7}, we obtain that 
\begin{align*}
&\mathcal Y_n(\underbrace{2,\dots,2}_{m-\ell}\underbrace{1,\dots,1}_{\ell})
+n\sum_{j=0}^{n-m-2}\mathcal Y_n(\underbrace{3,\dots,3}_{m-\ell+1+j},\underbrace{2,\dots,2}_{\ell-1-j},\underbrace{1,\dots,1}_{n-m-2-j})\\
&=\frac{1}{n(m+1)}\left(\binom{n-1}{m}+(-1)^m\binom{n-1}{2 m+1}\right)\binom{n}{\ell}\,. 
\end{align*}
\end{proof}

If the parts of powers are interchanged, then we also have a different but similar result.  

\begin{theorem}  
For integers $n$, $m$ and $\ell$ with $n-1\ge m\ge\ell\ge 0$, we have 
\begin{align*}
&n\mathcal Y_n(\underbrace{2,\dots,2}_{m-\ell}\underbrace{1,\dots,1}_{n-m-\ell-1})
+n^2\sum_{j=0}^{n-m-2}\mathcal Y_n(\underbrace{3,\dots,3}_{m-\ell+1+j},\underbrace{2,\dots,2}_{n-m-2-j},\underbrace{1,\dots,1}_{\ell-1-j})\\
&=\frac{1}{m+1}\binom{n-1}{m}\left(\binom{n}{n-\ell}+2(-1)^{n-\ell-1}\binom{n}{2 n-2 \ell}\right)\,. 
\end{align*} 
\label{th:9}
\end{theorem}

This is the immediate consequence of the following. Together with the identity in Theorem \ref{th:-12} (\ref{eq:-2}), we can prove Theorem \ref{th:9}.   

\begin{Prop}   
For integers $n$, $m$ and $\ell$ with $n-1\ge m\ge\ell\ge 0$, we have  
\begin{align*}
&\mathfrak Z_n(\zeta_n;;\underbrace{1,\dots,1}_m)\mathfrak Z_n(\zeta_n;;\underbrace{-2,\dots,-2}_\ell)\\
&=\frac{1}{u_1\cdots u_{n-1}}\mathcal Y_n(\underbrace{2,\dots,2}_{m-\ell}\underbrace{1,\dots,1}_{n-m-\ell-1})\\
&\quad +\frac{1}{u_1^2\cdots u_{n-1}^2}\sum_{j=0}^{n-m-2}\mathcal Y_n(\underbrace{3,\dots,3}_{m-\ell+1+j},\underbrace{2,\dots,2}_{n-m-2-j},\underbrace{1,\dots,1}_{\ell-1-j})\,. 
\end{align*}
\label{prp:9}
\end{Prop}

\section{Products of the multiple zeta values with the same parity} 

We shall give the results on products of the multiple zeta values with the same parity. Here, when the length of the indices (powers) is negative, $\mathcal Y_n$ is vanished.  

\begin{Prop}
For positive integers $n$, $m$ and $\ell$, we have 
$$
\mathfrak Z_n(\zeta_n;;\underbrace{1,\dots,1}_m)\mathfrak Z_n(\zeta_n;;\underbrace{1,\dots,1}_\ell)=\sum_{j=0}^{n-m-1}\binom{m-\ell+2 j}{j}\mathcal Y_n(\underbrace{2,\dots,2}_{\ell-j},\underbrace{1,\dots,1}_{m-\ell+2 j})\,. 
$$ 
\label{prp:ml11}  
\end{Prop}  
\begin{proof}
Without loss of generality, we can assume that $m\ge\ell$. 
Suppose that $\{u_{k_1},\dots,u_{k_m}\}$ are from $\mathfrak Z_n(\zeta_n;;\underbrace{1,\dots,1}_m)$ and $\{u_{j_1},\dots,u_{j_\ell}\}$ are from $\mathfrak Z_n(\zeta_n;;\underbrace{1,\dots,1}_\ell)$.  
 
If $\{u_{j_1},\dots,u_{j_\ell}\}\subseteq\{u_{k_1},\dots,u_{k_m}\}\subseteq\{u_1,\dots,u_{n-1}\}$, then $\ell$ elements are overlapped and the rest of $m-\ell$ elements are single, yielding that the string of powers $\underbrace{2,\dots,2}_\ell,\underbrace{1,\dots,1}_{m-\ell}$, which total length is $m$.    

If $j$ elements from $\{u_{j_1},\dots,u_{j_\ell}\}$ does not belong to $\{u_{k_1},\dots,u_{k_m}\}$, the total length of string is $m+j$, and the length of overlapped part is $\ell-j$, the single part is $(m+j)-(\ell-j)=m-\ell+2 j$, yielding $\underbrace{2,\dots,2}_{\ell-j},\underbrace{1,\dots,1}_{m-\ell+2 j}$. The same pattern can be obtained in $\binom{m-\ell+2 j}{j}$ ways because it depends on whether the part that does not overlap belong to only $\{u_{j_1},\dots,u_{j_\ell}\}$ or only $\{u_{k_1},\dots,u_{k_m}\}$.  

For example, if $\{u_{k_1},\dots,u_{k_m}\}=\{u_1,u_2,u_3,u_4,u_5,u_6\}$ and $\{u_{j_1},\dots,u_{j_\ell}\}=\{u_5,u_6,u_7,u_8\}$, then this have the expression $u_1 u_2 u_3 u_4 u_5^2 u_6^2 u_7 u_8$. But this same expression can be yielded from $\{u_2,u_3,u_4,u_5,u_6,u_7\}$ and $\{u_1,u_5,u_6,u_8\}$. Nevertheless, $\{u_5,u_6\}$ are common and fixed, and the elements of the single part can be chosen in $\binom{6}{2}=\binom{6-4+2\cdot 2}{2}$ ways.   

Since the value of $j$ can go up to the smaller of $n-m-1$ (by considering $\{u_{k_1},\dots,u_{k_m}\}$) and $\ell$ (by considering $\{u_{j_1},\dots,u_{j_\ell}\}$), the claim holds.       
\end{proof}

By applying (\ref{eq:zz-1m}), from Proposition \ref{prp:ml11}, we can get the following explicit expression.  

\begin{theorem} 
For positive integers $n$, $m$ and $\ell$, we have 
$$ 
\sum_{j=0}^{n-m-1}\binom{m-\ell+2 j}{j}\mathcal Y_n(\underbrace{2,\dots,2}_{\ell-j},\underbrace{1,\dots,1}_{m-\ell+2 j})=
\frac{1}{(m+1)(\ell+1)}\binom{n-1}{m}\binom{n-1}{\ell}\,. 
$$ 
\label{th:ml11}
\end{theorem}

By further applying Theorem \ref{th:ml11}, we have finally succeeded in obtaining an explicit formula for the sum of finite $q$-multiple zeta values ($q$-multiple harmonic sums) of the type with exponents $2-\cdots-2,1-\cdots-1$.

\begin{theorem}  
For an integer $r$ with $0\le r\le m$, we have 
$$
\mathcal Y_n(\underbrace{2,\dots,2}_{r},\underbrace{1,\dots,1}_{m-r})
=\frac{1}{(r+1)n}\binom{m}{r}\left(\binom{n-1}{m}+(-1)^r\binom{n-1}{m+r+1}\right)\,.
$$
\label{th:222111}
\end{theorem}

\noindent 
{\it Remark.}  
When $r=0$ in Theorem \ref{th:222111}, we have 
$$
\mathcal Y_n(\underbrace{1,\dots,1}_{m})=\frac{1}{n}\binom{n}{m+1}=\frac{1}{m+1}\binom{n-1}{m}=\mathfrak Z_n(\zeta_n;m,1)\,,
$$
which is \cite[Theorem 9]{Ko24}. 

When $r=1$ in Theorem \ref{th:222111}, we have 
$$
\mathcal Y_n(2,\underbrace{1,\dots,1}_{m-1})=m\cdot\frac{-(n-2 m-3)}{2(m+2)(m+1)}\binom{n-1}{m}=m\times \mathfrak{Re}\left(\mathfrak Z_n(\zeta_n;;\underbrace{\underbrace{1,\dots,1}_a,2,\underbrace{1,\dots,1}_b}_m)\right)\,, 
$$ 
which is equivalent to \cite[Theorem 1]{BIK}, as shown in (\ref{eq:11112}) or (\ref{eq:11112r}).  

When $r=m-1$ in Theorem \ref{th:222111}, we have 
\begin{equation}
\mathcal Y_n(\underbrace{2,\dots,2}_{m-1},1)
=\frac{1}{n}\left(\binom{n-1}{m}+(-1)^{m-1}\binom{n-1}{2 m}\right)\,.
\label{eq:222221}
\end{equation}

When $r=m$ in Theorem \ref{th:222111}, we have 
$$
\mathcal Y_n(\underbrace{2,\dots,2}_{m})=\frac{1}{(m+1)n}\left(\binom{n-1}{m}+(-1)^m\binom{n-1}{2 m+1}\right)\,,
$$
which is \cite[Theorem 5]{Ko25a}.

\begin{proof} 
As shown in Remark above, the identity in Theorem \ref{th:222111} is valid for $r=0$ (and $r=1$).   
Assume that the identity in Theorem \ref{th:222111} is valid for all $r$ with $0\le r\le\ell-1$ ($\ell\ge 1$).  Then by applying Theorem \ref{th:ml11}, (note that $\mathcal Y_n$ is nullified if the length of the string is negative.) we have
 
\begin{align}
&\mathcal Y_n(\underbrace{2,\dots,2}_{\ell},\underbrace{1,\dots,1}_{m-\ell})\notag\\
&=-\sum_{j=1}^{\ell}\binom{m-\ell+2 j}{j}\mathcal Y_n(\underbrace{2,\dots,2}_{\ell-j},\underbrace{1,\dots,1}_{m-\ell+2 j})
+\frac{1}{(m+1)(\ell+1)}\binom{n-1}{m}\binom{n-1}{\ell}\notag\\
&=-\sum_{j=1}^{\ell}\binom{m-\ell+2 j}{j}\frac{1}{(\ell-j+1)n}\binom{m+j}{\ell-j}\left(\binom{n-1}{m+j}+(-1)^{\ell-j}\binom{n-1}{m+\ell+1}\right)\notag\\
&\quad +\frac{1}{(\ell+1)n}\binom{n-1}{\ell}\binom{n}{m+1}\,. 
\label{eq:2222211111}
\end{align}
Here, we used the assumption for the identity in Theorem \ref{th:222111} by replacing $m$ by $m+j$.
     
Since 
\begin{align*}
\binom{m-\ell+2 j}{j}\frac{\ell+1}{(\ell-j+1)}\binom{m+j}{\ell-j}
&=\frac{(\ell+1) (m+j)!}{j!(m-\ell+j)! (\ell-j+1)! m!}\\
&=\binom{\ell+1}{j}\binom{m+j}{\ell}\,,
\end{align*}
the first term of the right-hand side of (\ref{eq:2222211111}) is written as 
$$
\frac{S_1+S_2}{(\ell+1)n}\,,
$$
where 
$$
S_1:=-\sum_{j=1}^{\ell}\binom{\ell+1}{j}\binom{m+j}{\ell}\binom{n-1}{m+j}
$$
and 
$$
S_2:=-\binom{n-1}{m+\ell+1}\sum_{j=1}^{\ell}(-1)^{\ell-j}\binom{\ell+1}{j}\binom{m+j}{\ell}\,.
$$ 
Since by the Chu--Vandermonde identity, 
\begin{align*} 
&\sum_{j=0}^{\ell}\binom{\ell+1}{j}\binom{m+j}{\ell}\binom{n-1}{m+j}\\
&=\frac{\ell+1}{n}\binom{n}{m+1}\sum_{j=0}^\ell\binom{m+1}{\ell-j+1}\binom{n-m-1}{j}\\
&=\frac{\ell+1}{n}\binom{n}{m+1}\left(\sum_{j=0}^{\ell+1}\binom{m+1}{\ell-j+1}\binom{n-m-1}{j}-\binom{n-m-1}{\ell+1}\right)\\
&=\frac{\ell+1}{n}\binom{n}{m+1}\left(\binom{n}{\ell+1}-\binom{n-m-1}{\ell+1}\right)\,,
\end{align*}
we have 
\begin{align*}
S_1&=-\sum_{j=0}^{\ell}\binom{\ell+1}{j}\binom{m+j}{\ell}\binom{n-1}{m+j}+\binom{m}{\ell}\binom{n-1}{m}\\
&=-\frac{\ell+1}{n}\binom{n}{m+1}\left(\binom{n}{\ell+1}-\binom{n-m-1}{\ell+1}\right)+\binom{m}{\ell}\binom{n-1}{m}\,.
\end{align*}
Using the identity 
$$
\sum_{j=0}^{\ell}(-1)^{\ell-j}\binom{\ell+1}{j}\binom{m+j}{\ell}
=\binom{m+\ell+1}{\ell} 
$$
and subtracting the term for $j=0$, we have 
$$
S_2=-\binom{n-1}{m+\ell+1}\left(\binom{m+\ell+1}{\ell}-(-1)^\ell\binom{m}{\ell}\right)\,.
$$ 
Therefore, 
the right-hand side of (\ref{eq:2222211111}) is equal to 
\begin{align*}
&\frac{S_1+S_2}{(\ell+1)n}+\frac{1}{(\ell+1)n}\binom{n-1}{\ell}\left(\binom{n-1}{m}+\binom{n-1}{m+1}\right)\\
&=-\frac{1}{n^2}\binom{n}{m+1}\left(\binom{n}{\ell+1}-\binom{n-m-1}{\ell+1}\right)+\frac{1}{(\ell+1)n}\binom{m}{\ell}\binom{n-1}{m}\\
&\quad-\frac{1}{(\ell+1)n}\binom{n-1}{m+\ell+1}\left(\binom{m+\ell+1}{\ell}-(-1)^\ell\binom{m}{\ell}\right)\\
&\quad +\frac{1}{(\ell+1)n}\binom{n-1}{\ell}\binom{n}{m+1}\\ 
&=\frac{1}{(\ell+1)n}\binom{m}{\ell}\left(\binom{n-1}{m}+(-1)^\ell\binom{n-1}{m+\ell+1}\right)\,.
\end{align*} 
Notice that all other terms are vanished because  
\begin{align*}
&-(\ell+1)\binom{n}{m+1}\left(\binom{n}{\ell+1}-\binom{n-m-1}{\ell+1}\right)\\
&\quad 
-n\binom{n-1}{m+\ell+1}\binom{m+\ell+1}{\ell}+n\binom{n-1}{\ell}\binom{n}{m+1}\\
&=-\frac{n!}{(m+1)!(n-m-1)!}\left(\frac{n!}{\ell!(n-\ell-1)!}-\frac{(n-m-1)!}{\ell!(n-m-\ell-2)!}\right)\\
&\quad +\frac{n!}{\ell!(m+1)!}\left(-\frac{1}{(n-m-\ell-2)!}+\frac{n!}{(n-\ell-1)!(n-m-1)!}\right)\\
&=0\,. 
\end{align*} 
Hence, the identity in Theorem \ref{th:222111} is also valid for $r=\ell$.  
By induction, the desired result has been proved.  
\end{proof}

\subsection{Products of the multiple zeta values with the same negative parity}  
Finally we shall consider the product of multiple zeta values with negative powers.  

\begin{Prop}
For positive integers $n$, $m$ and $\ell$, we have 
$$
\mathfrak Z_n(\zeta_n;;\underbrace{-1,\dots,-1}_m)\mathfrak Z_n(\zeta_n;;\underbrace{-1,\dots,-1}_\ell)=n^2\sum_{j=0}^{n-m-1}\binom{m-\ell+2 j}{j}\mathcal Y_n(\underbrace{2,\dots,2}_{n-m-j-1},\underbrace{1,\dots,1}_{m-\ell+2 j})\,. 
$$ 
\label{prp:ml-1-1}  
\end{Prop}  
\begin{proof}
From Proposition \ref{prp:ml11}, the same strings appear in the denominators.  By multiplying $(u_1\cdots u_{n-1})^2$, the same string of $1-\cdots-1$ appears as they are, and the string of $2-\cdots-2$ disapper. In addition, the hidden part appears as the string of $2-\cdots-2$, and such a length is $n-1-(\ell-j)-(m-\ell+2 j)=n-m-j-1$. 

For example, when $n=8$, 
$$
\frac{1}{u_1^2 u_2^2 u_3 u_4 u_5}=\frac{u_4 u_5 u_6^2 u_7^2}{(u_1\cdots u_7)^2}
=8^2(u_4 u_5 u_6^2 u_7^2)\,.
$$ 
Hence, by Proposition \ref{prp:ml11}, we have 
$$
\mathfrak Z_n(\zeta_n;;\underbrace{-1,\dots,-1}_m)\mathfrak Z_n(\zeta_n;;\underbrace{-1,\dots,-1}_\ell)=n^2\sum_{j=0}^{n-m-1}\binom{m-\ell+2 j}{j}\mathcal Y_n(\underbrace{2,\dots,2}_{n-m-j-1},\underbrace{1,\dots,1}_{m-\ell+2 j})\,. 
$$  
\end{proof}

From Proposition \ref{prp:ml-1-1}, by using Theorem \ref{th:-12} (\ref{eq:-1}), we immediately have the following.   

\begin{theorem}
We have 
$$
\sum_{j=0}^{n-m-1}\binom{m-\ell+2 j}{j}\mathcal Y_n(\underbrace{2,\dots,2}_{n-m-j-1},\underbrace{1,\dots,1}_{m-\ell+2 j})=\frac{1}{n^2}\binom{n}{m}\binom{n}{\ell}\,.
$$ 
\label{th:ml-1-1}
\end{theorem}

\section{Comments and future works}   
     
In this paper, we have succeeded in providing several explicit formulas in the form of polynomials for finite $q$-multiple zeta values with different indices, including the type $2-\cdots-2,1-\cdots-1$. This is a sum that includes all permutations, and it is not an explicit formula for each multiple zeta value itself. The average sum is frequently discussed in analytic number theory and is significant (see, e.g., \cite{HR92,Rosen02}), but ultimately it is also interesting to see what an explicit formula for each multiple zeta value itself would look like. However, since each multiple zeta value also has an imaginary part, handling them is currently very difficult.

\section*{Competing Interests}  
The authors declare no competing interest.




\begin{thebibliography}{100}

\bibitem{BTT18} 
H. Bachmann, Y. Takeyama and K. Tasaka, {\em 
Cyclotomic analogues of finite multiple zeta values},  
Compositio Math. {\bf 154} (2018), 2701--2721. 

\bibitem{BTT20} 
H. Bachmann, Y. Takeyama and K. Tasaka, {\em 
Special values of finite multiple harmonic $q$-series at roots of unity},  
Algebraic combinatorics, resurgence, moulds and applications (CARMA). Vol. 2, 1--18.
IRMA Lect. Math. Theor. Phys., 32; 
EMS Publishing House, Berlin, 2020. 


\bibitem{BIK}  
Y. Bilu, H. Ishikawa and T. Komatsu,  {\em 
Some explicit values of a $q$-multiple zeta function whose denominator power is not uniform},  
arXiv:2512.06672 (2025).  

\bibitem{Bradley}
D. M. Bradley, {\em  
Multiple $q$-zeta values}, 
J. Algebra {\bf 283}(2) (2005), 752--798. 



\bibitem{CK}  
T. Chatterjee and T. Komatsu, {\em  
Special values of a $q$-multiple t-function of general level at roots of unity}, 
Mediterr. J. Math. {\bf 22} (2025), No.8, Article 226, 20 p.  
DOI: 10.1007/s00009-025-02993-1
 



\bibitem{HR92} 
J. Hoffstein and M. Rosen, {\em 
Average values of L-series in function fields}, 
J. Reine Angew. Math. {\bf 426} (1992), 117--150. 


\bibitem{Ko23}   
T. Komatsu, {\em 
On $s$-Stirling transform and poly-Cauchy numbers of the second kind with level $2$}, 
Aequationes Math. {\bf 97} (2023), no.1, 31--61.

\bibitem{Ko24}  
T. Komatsu, {\em  
On $q$-generalized $(r,s)$-Stirling numbers},  
Aequationes Math. {\bf 98} (2024), 1281--1304. 

\bibitem{Ko25b}
T. Komatsu, {\em  
Some explicit values of a $q$-multiple zeta-star function at roots of unity}, 
RAIRO, Theor. Inform. Appl. {\bf 59} (2025), Paper No.5, 14 p.  

\bibitem{Ko25a}  
T. Komatsu, {\em  
Some explicit values of a $q$-multiple zeta function at roots of unity}, 
J. Math. Anal. Appl. {\bf 555} (2026), No.2, Article 130065, 16 p. 

\bibitem{KL}  
T. Komatsu and F. Luca, {\em  
Some explicit forms of special values of an alternating $q$-multiple $t$-function of general level at roots of unity}, 
Ramanujan J. {\bf 68} (2025), Article no.3, 25 p.   
DOI: 10.1007/s11139-025-01150-2
  
\bibitem{KP}
T. Komatsu and R. K. Pandey, {\em  
Some explicit values of an alternative q-multiple zeta function at roots of unity}, 
Log. J. IGPL {\bf 33} (2025), No.6, Article ID jzaf075, 19 p.  
DOI: 10.1093/jigpal/jzaf075

\bibitem{KW}  
T. Komatsu and T. Wang, {\em  
Some explicit values of a q-multiple t-function at roots of unity}, 
Aequationes Math. (online first). 
DOI:10.1007/s00010-025-01235-9


\bibitem{OOZ}  
Y. Ohno, J. Okuda and W. Zudilin, {\em  
Cyclic $q$-MZSV sum}, 
J. Number Theory {\bf 132}(1) (2012), 144--155.

\bibitem{Rosen02}
M. Rosen, {\em 
Number theory in function fields}, 
Graduate Texts in Mathematics 210. New York, NY: Springer, 2002. 


\bibitem{Schlesinger} 
K.-G. Schlesinger, {\em 
Some remarks on $q$-deformed multiple polylogarithms}, 
arXiv:math/0111022, Nov. 2001. 

\bibitem{Takeyama09}  
Y. Takeyama, {\em 
A  $q $-analogue of non-strict multiple zeta values and basic hypergeometric series}, 
Proc. Amer. Math. Soc. {\bf 137} (2009), no. 9, 2997--3002.

\bibitem{Tasaka21} 
K. Tasaka, {\em 
Finite and symmetric colored multiple zeta values and multiple harmonic $q$
-series at roots of unity}, 
Sel. Math., New Ser. {\bf 27} (2021), No. 2, Paper No. 21, 34 p.. 

\bibitem{Zhao} 
J. Zhao,  {\em  
Multiple $q$-zeta functions and multiple $q$-polylogarithms},   
Ramanujan J. {\bf 14}(2) (2007), 189--221. 

\bibitem{Zhao_book}
J. Zhao,  {\em  
Multiple zeta functions, multiple polylogarithms and their special values}, 
Series on Number Theory and Its Applications 12, Hackensack, NJ; World Scientific, 2016. 

\bibitem{Zudilin}
W. Zudilin, {\em 
Algebraic relations for multiple zeta values} (Russian),  
Uspekhi Mat. Nauk {\bf 58}(1) (2003), 3--32; 
translation in Russian Math. Surveys {\bf 58}(1) (2003), 1--29.

\end{thebibliography}
\end{document}